\documentclass[12 pt]{amsart}

\newtheorem{theorem}{Theorem}[section]
\newtheorem{lemma}[theorem]{Lemma}
\newtheorem{proposition}[theorem]{Proposition}

\theoremstyle{definition}

\usepackage{amscd,amssymb}

\begin{document}

\title[Automorphisms and Dirichlet series]
{Automorphisms of polynomial algebras and Dirichlet series}

\author[Vesselin Drensky and Jie-Tai Yu]
{Vesselin Drensky and Jie-Tai Yu}
\address{Institute of Mathematics and Informatics,
Bulgarian Academy of Sciences,
1113 Sofia, Bulgaria}
\email{drensky@math.bas.bg}
\address{Department of Mathematics,
The University of Hong Kong,
Hong Kong SAR, China}
\email{yujt@hkucc.hku.hk}

\thanks
{The research of Vesselin Drensky was partially supported by Grant
MI-1503/2005 of the Bulgarian National Science Fund.}

\thanks
{The research of Jie-Tai Yu was partially supported by an RGC-CERG grant.}

\subjclass[2000] {11T06; 13B25; 05A15; 30B50; 16S10; 17A50.}
\keywords{Automorphisms, polynomial algebras, free associative
algebras, finite fields, coordinates,  Nielsen-Schreier varieties,
Dirichlet series, generating functions}

\begin{abstract}
Let ${\mathbb F}_q[x,y]$ be the polynomial algebra in two
variables over the finite field ${\mathbb F}_q$ with $q$ elements.
We give an exact formula and the asymptotics
for the number $p_n$ of automorphisms $(f,g)$ of
${\mathbb F}_q[x,y]$ such that $\max\{\text{deg}(f),\text{deg}(g)\}=n$.
We describe also the Dirichlet series generating function
\[
p(s)=\sum_{n\geq 1}\frac{p_n}{n^s}.
\]
The same results hold for the automorphisms of the free associative algebra
${\mathbb F}_q\langle x,y\rangle$. We have also obtained analogues for free
algebras with two generators in Nielsen -- Schreier varieties of algebras.
\end{abstract}

\maketitle

\section*{Introduction}

Our paper is devoted to the following problem. {\it Let ${\mathbb
F}_q[x,y]$ be the polynomial algebra in two variables over the
finite field ${\mathbb F}_q$ with $q$ elements.  We would like to
determine the number of $\mathbb F_q$-automorphisms $\varphi=(f,g)$
of ${\mathbb F}_q[x,y]$ satisfying}
$\deg(\varphi):=\max\{\text{deg}(f),\text{deg}(g)\}=n$. Here
$\varphi=(f,g)$ means that $f=\varphi(x)$, $g=\varphi(y)$.

Our consideration is motivated by Arnaud Bodin \cite{B}, who raised
the question  to determine the number of $\mathbb F_q$-automorphisms
$\varphi$ with $\deg(\varphi)\le n$.

In the sequel all automorphisms are $\mathbb F_q$-automorphisms.
The theorem of Jung -- van der Kulk \cite{J, K} states that
the automorphisms of the polynomial algebra $K[x,y]$ over any field $K$ are tame.
In other words the group $\text{Aut}(K[x,y])$ is generated by the subgroup
$A$ of affine automorphisms
\[
\alpha=(a_1x+b_1y+c_1,a_2x+b_2y+c_2),\quad a_i,b_i,c_i\in K,
\quad a_1b_2\not= a_2b_1,
\]
and the subgroup $B$ of triangular automorphisms
\[
\beta=(ax+h(y),by+b_1),\quad 0\not=a,b\in K, \quad b_1\in K,\quad h(y)\in K[y].
\]
The proof of van der Kulk \cite{K} gives that
$\text{Aut}(K[x,y])$ has the following nice structure,
see e.g. \cite{C}:
\[
\text{Aut}(K[x,y]) = A\ast_CB,\quad C=A\cap B,
\]
where $A\ast_CB$ is the free product of $A$ and $B$
with amalgamated subgroup $C=A\cap B$.
Using the canonical form of the elements of $\text{Aut}(K[x,y])$
we have calculated explicitly the number $p_n$ of
automorphisms of degree $n$:
\[
p_1=q^3(q-1)^2(q+1),
\]
\[
p_n=(q(q-1)(q+1))^2\sum \left(\frac{q-1}{q}\right)^kq^{n_1+\cdots+n_k},
\quad n>1,
\]
where the summation is on all ordered factorizations $n=n_1\cdots n_k$
of $n$, with $n_1,\ldots,n_k>1$.

It is natural to express the sequence $p_n$, $n=1,2,\ldots$,
in terms of its generating function. When the elements
$p_n$ of the sequence involve sums on the divisors of the index $n$
it is convenient to work with the Dirichlet series generating function,
i.e., with the formal series
\[
p(s)=\sum_{n\geq 1}\frac{p_n}{n^s}.
\]
For the Riemann zeta function $\zeta(s)$ the coefficients of
$\zeta^k(s)$ count the number of ordered factorizations $n=n_1\cdots n_k$ in $k$ factors.
(We have to take $(\zeta(s)-1)^k$ if we want to count only
factorizations with $n_i\geq 2$.) Similarly, the coefficients of the
$k$-th power $\rho^k(s)$ of the formal Dirichlet series
\[
\rho(s)=\sum_{n\geq 2}\frac{q^n}{n^s}
\]
are equal to $q^{n_1+\cdots+n_k}$ in the expression of $p_n$.
Hence $\rho(s)$ may be considered as a $q$-analogue of $\zeta(s)$,
although it does not satisfy many of the nice properties of the Riemann zeta function
(because the sequence $q^n$ is not multiplicative) and for $q>1$
is not convergent for any nonzero $s$ (because its coefficients grow faster than $n^s$).
We have found that
\[
p(s)=(q(q-1)(q+1))^2\left(\sum_{k\geq 0}\frac{q-1}{q}\rho^k(s)-\frac{1}{q+1}\right)
\]
\[
=(q(q-1)(q+1))^2\left(\frac{1}{1-\frac{q-1}{q}\rho(s)}-\frac{1}{q+1}\right)
\]
and have given an estimate for the growth of $p_n$. For $n\geq 2$
\[
(q-1)^3(q+1)^2q^{n+1}\leq p_n\leq (q-1)^3(q+1)^2q^{n+1}+(\log_2n)^{\log_2n}q^{n/2+8}.
\]
Hence for a fixed $q$ and any $\varepsilon>0$,
\[
p_n=(q-1)^3(q+1)^2q^{n+1}+{\mathcal O}(q^{n(1/2+\varepsilon)}).
\]
The main contribution
$(q-1)^3(q+1)^2q^{n+1}$ to $p_n$ comes from
the number of automorphisms of the form
\[
(a_1x+b_1y+c_1+a_1h(y),a_2x+b_2y+c_2+a_2h(y)),
\]
\[
(b_1x+(a_1+ab_1)y+c_1+a_1h(x+ay),b_2x+(a_2+ab_2)y+c_2+a_2h(x+ay)),
\]
where $a,a_i,b_i,c_i\in{\mathbb F}_q$, $a_1b_2\not=a_2b_1$,
\[
h(y)=h_ny^n+h_{n-1}y^{n-1}+\cdots+h_2y^2\in y^2{\mathbb F}_q[y],\quad h_n\not=0.
\]

Based on the above result, we  have also calculated explicitly the
number $l_n$ of coordinates of degree $n$ and obtained its Dirichlet
series generating function, i.e., with the formal series
\[
l(s)=\sum_{n\geq 1}\frac{l_n}{n^s}.
\]

By the theorem of Czerniakiewicz and Makar-Limanov \cite{Cz, ML}
for the tameness of the automorphisms of $K\langle x,y\rangle$ over any field $K$
and the isomorphism $\text{Aut}(K[x,y])\cong \text{Aut}(K\langle x,y\rangle)$
which preserves the degree of the automorphisms
we derive immediately that the same results hold for the number of automorphisms of
degree $n$ of the free associative algebra
${\mathbb F}_q\langle x,y\rangle$.

It is easy to obtain an analogue of the formula
for the number of automorphisms $p_n$ of degree $n$ for free algebras with two
generators over ${\mathbb F}_q$ if the algebra satisfies the Nielsen -- Schreier
property. Examples of such algebras
are the free Lie algebra and the free anti-commutative algebra
(where $p_1=q(q-1)^2(q+1)$ and $p_n=0$ for $n>1$), free nonassociative algebras and
free commutative algebras.

In a forthcoming paper we are going to give an algebraic geometrical analogue
of the main results of the present paper for infinite fields.

\section{Canonical forms of automorphisms}

The group $G$ is the free product of its subgroups $A$ and $B$
with amalgamated subgroup $C=A\cap B$ (notation $G=A\ast_CB$), if
$G$ is generated by $A$ and $B$ and if for any $a_1,\ldots,a_{k+1}\in A$,
$b_1,\ldots,b_k\in B$, $k\geq 1$,
such that $a_2,\ldots,a_k,b_1,\ldots,b_k$ do not belong to $C$, the product
$g=a_1b_1\cdots a_kb_ka_{k+1}$ does not belong to $C$.

For the following description of $G=A\ast_CB$ see e.g.
\cite{MKS}, p. 201, Theorem 4.4 and its corollaries.

\begin{lemma}\label{description of free product}
Let $G=A\ast_CB$, $C=A\cap B$, and let
\[
A_0=\{1,a_i\in A\mid i\in I\},\quad
B_0=\{1,b_j\in B\mid j\in J\}
\]
be, respectively, left coset representative systems for $A$ and $B$ modulo $C$.
Then each $g\in G$ can be presented in a unique way in the form
\[
g=g_1\cdots g_kc,
\]
where $1\not=g_i\in A_0\cup B_0$, $i=1,\ldots,k$,
$g_i,g_{i+1}$ are neither both in $A_0$, nor both in $B_0$, $c\in C$.
\end{lemma}

We write the automorphisms
of $K[x,y]$ over any field $K$ as functions. If
$\varphi=(f_1(x,y),g_1(x,y))$, $\psi=(f_2(x,y),g_2(x,y))$, then
$\varphi\circ\psi(u)=\varphi(\psi(u))$, $u\in K[x,y]$, and hence
\[
\varphi\circ\psi=(f_2(f_1(x,y),g_1(x,y)),g_2(f_1(x,y),g_1(x,y))).
\]
The following presentation of the automorphisms of $K[x,y]$ is
well known, see e.g. Wright \cite{Wr}.
We include the proof for self-containess of the exposition.

\begin{proposition}\label{canonical form of automorphisms}
Define the sets of automorphisms of $K[x,y]$
\[
A_0=\{\iota=(x,y),\quad \alpha=(y,x+ay)\mid a\in K\},
\]
\[
B_0=\{\beta=(x+h(y),y)\mid h(y)\in y^2K[y]\}.
\]
Every automorphism $\varphi$ of $K[x,y]$ can be presented in a unique
way as a composition
\[
\varphi=(f,g)=\alpha_1\circ\beta_1\circ\alpha_2\circ\beta_2\circ
\cdots\circ\alpha_k\circ\beta_k\circ\lambda,
\]
where $\alpha_i\in A_0$, $\alpha_2,\ldots,\alpha_k\not=\iota$,
$\beta_i\in B_0$, $\beta_1,\ldots,\beta_k\not=\iota$, $\lambda\in A$.
If $\beta_i=(x+h_i(y),y)$ and $\text{\rm deg}(h_i(y))=n_i$,
then the degree of $\varphi$
\[
n=\text{\rm deg}(\varphi)=\max\{\text{\rm deg}(f),\text{\rm deg}(g)\}=n_1\cdots n_k
\]
is equal to the product of the degrees of $\beta_i$.
\end{proposition}

\begin{proof}
First we shall show that $A_0$ and $B_0$ are, respectively,
left coset representative systems for the group $A$ of affine automorphisms
and the group $B$ of triangular automorphisms modulo the intersection $C=A\cap B$.

Since $\iota$ is the identity automorphism, $\alpha=(y,x+ay)\not\in C$ and
\[
\alpha_1^{-1}\circ\alpha_2=(y,x+a_1y)^{-1}\circ(y,x+a_2y)
=(x,(-a_1+a_2)x+y)
\]
does not belong to $C$ when $a_1\not=a_2$,
to show the statement for $A_0$ it is sufficient to verify that
for any $\lambda\in A$ there exist $\alpha\in A_0$ and $\gamma\in C$ such that
$\lambda=\alpha\circ \gamma$. We choose $\alpha=\iota$ and $\gamma=\lambda$
if $\lambda\in C$ and have the presentation
\[
\lambda=(a_1x+b_1y+c_1,a_2x+b_2y+c_2)
\]
\[
=(y,x+\frac{b_2}{a_2}y)\circ ((b_1-\frac{a_1b_2}{a_2})x+a_1y+c_1,a_2y+c_2)
\]
when $\lambda\not\in C$ (and hence $a_2\not=0$). Similarly, for $B_0$
it is sufficient to use that $\beta=(x+h(y),y)\not\in C$
when $0\not=h(y)\in y^2K[y]$,
\[
\beta_1^{-1}\circ\beta_2=(x+h_1(y),y)^{-1}\circ(x+h_2(y),y)
=(x-h_1(y)+h_2(y),y)\not\in C
\]
for $h_1(y)\not=h_2(y)$ (because $h_1$ and $h_2$ have no monomilas of degree $<2$)
and to see that
\[
(ax+h(y),by+b_1)=(x+k(y),y)\circ(ax+h_1y+h_0,by+b_1)\in B_0\circ C,
\]
where $a,b,b_1\in K$, $a,b\not=0$, $h(y)=h_ny^n+\cdots+h_1y+h_0$ and
\[
k(y)=\frac{1}{a}(h_ny^n+\cdots+h_3y^3+h_2y^2).
\]
Lemma \ref{description of free product} gives that
every automorphism $\varphi$ of $K[x,y]$ has a unique presentation
\[
\varphi=(f,g)=\alpha_1\circ\beta_1\circ\alpha_2\circ\beta_2\circ
\cdots\circ\alpha_k\circ\beta_k\circ\alpha_{k+1}\circ\gamma,
\]
where $\alpha_i\in A_0$, $\beta_i\in B_0$,
$\beta_1,\alpha_2,\ldots,\alpha_k,\beta_k\not=\iota$, $\gamma\in C$. Since
$\lambda=\alpha_{k+1}\circ\gamma\in A$, we obtain the presentation.

Finally, if
\[
\lambda=(a_1x+b_1y+c_1,a_2x+b_2y+c_2),\quad \beta=(x+h(y),y),
\quad \alpha=(y,x+ay),
\]
$h=h_ny^n+\cdots+h_2y^2$, $h_n\not=0$, then the degree of
$\beta\circ\lambda$ and $\alpha\circ\beta\circ\lambda$ is $n$ and
homogeneous components of maximal degree of these automorphisms
are, respectively,
\[
\overline{\beta\circ\lambda}=(a_1h_ny^n,a_2h_ny^n),\quad
\overline{\alpha\circ\beta\circ\lambda}=(a_1h_n(x+ay)^n,a_2h_n(x+ay)^n).
\]
They are different from 0 because $(a_1,a_2)\not=(0,0)$.
If the homogeneous component of maximal degree of
$\psi=\alpha_i\circ\beta_i\circ\cdots\circ\alpha_k\circ\beta_k\circ\lambda$ is
\[
\overline{\psi}=(d_1(x+a_0y)^m,d_2(x+a_0y)^m),
\]
then for $\beta\circ\psi$ and $\alpha\circ\beta\circ\psi$ we obtain
\[
\overline{\beta\circ\psi}=(d_1h_n^my^{mn},d_2h_n^my^{mn}),
\]
\[
\overline{\alpha\circ\beta\circ\psi}=(d_1h_n^m(x+ay)^{mn},d_2h_n^m(x+ay)^{mn})
\]
and we apply induction.
\end{proof}

\

\section{The main results}

Now we give a formula for the number of automorphisms of degree $n$ for
${\mathbb F}_q[x,y]$ and the corresponding Dirichlet series generating function.

\begin{theorem}\label{explicit formula and Dirichlet series}
{\rm (i)} The number $p_n$ of automorphisms $\varphi=(f,g)$ of degree $n$ of
${\mathbb F}_q[x,y]$, i.e., such that
\[
n=\max\{\text{\rm deg}(f),\text{\rm deg}(g)\}
\]
is given by the formulas
\[
p_1=q^3(q-1)^2(q+1),
\]
\[
p_n=(q(q-1)(q+1))^2\sum \left(\frac{q-1}{q}\right)^kq^{n_1+\cdots+n_k},
\quad n>1,
\]
where the summation is on all ordered factorizations $n=n_1\cdots n_k$
of $n$, with $n_1,\ldots,n_k>1$.

{\rm (ii)} The Dirichlet series generating function $p(s)$ of the sequence
$p_n$, $n=1,2,\ldots$, is
\[
p(s)=\sum_{n\geq 1}\frac{p_n}{n^s}
=(q(q-1)(q+1))^2\left(\frac{1}{1-\frac{q-1}{q}\rho(s)}-\frac{1}{q+1}\right),
\]
where
\[
\rho(s)=\sum_{n\geq 2}\frac{q^n}{n^s}.
\]
\end{theorem}

\begin{proof}
(i) Applying Proposition \ref{canonical form of automorphisms},
$p_n$ is a sum on all ordered factorizations $n=n_1\cdots n_k$, $n_i>1$,
of the number of automorphisms of the form
\[
\varphi=(f,g)=\alpha_1^{\delta}\circ\beta_1\circ\alpha_2\circ\beta_2\circ
\cdots\circ\alpha_k\circ\beta_k\circ\lambda,
\]
where: $\alpha_i=(y,x+a_iy)$, $a_i\in{\mathbb F}_q$, $\delta$ is 1
or 0, depending on whether or not $\alpha_1$ participates in the
decomposition of $\varphi$; $\beta_i=(x+h_i(y)£¬ y)$, $h_i(y)\in
y^2{\mathbb F}_q[y]$, $\text{deg}(h_i)=n_i$; $\lambda$ is an affine
automorphism. We have $q+1=\vert{\mathbb F}_q\vert+1$ possibilities
for $\alpha_1^{\delta}$ and $q=\vert{\mathbb F}_q\vert$
possibilities for the other $\alpha_i$. The number of polynomials
$h_{n_i}(y)=h_{n_i,i}y^{n_i}+h_{n_i-1,i}y^{n_i-1}\cdots+h_{n_i,2}y^2$
of degree $n_i$ is $(q-1)q^{n_i-2}$ (because $h_{n_i,i}\not=0$).
Finally, the cardinality of the affine group is
\[
\vert A\vert=(q^2-1)(q^2-q)q^2=q^3(q-1)^2(q+1):
\]
If $\lambda=(a_1x+b_1y+c_1,a_2x+b_2y+c_2)$, then we have
$q^2-1$ possibilities for the nonzero element $a_1x+b_1y$,
$q^2-q$ possibilities
for $a_2x+b_2y$ which is linearly independent with
$a_1x+b_1y$ and $q^2$ possibilities for $c_1,c_2$.
Hence, for $n=1$
\[
p_1=\vert A\vert=q^3(q-1)^2(q+1),
\]
\[
p_n=\sum_{n_1\cdots n_k=n}
(q+1)q^{k-1}\left(\prod_{i=1}^k(q-1)q^{n_i-2}\right)q^3(q-1)^2(q+1)
\]
\[
=(q+1)^2\sum_{n_1\cdots n_k=n}(q-1)^{k+2}q^{n_1+\cdots+n_k-k+2}
\]
\[
=(q(q-1)(q+1))^2\sum_{n_1\cdots n_k=n}\left(\frac{q-1}{q}\right)^kq^{n_1+\cdots+n_k}
\]
when $n>1$.

(ii) If $a_n$, $b_n$, $n=1,2,\ldots$, are two sequences, then, see e.g. \cite{W},
the product of their Dirichlet series generating functions $a(s)$, $b(s)$ is
\[
a(s)b(s)=\left(\sum_{n\geq 1}\frac{a_n}{n^s}\right)
\left(\sum_{n\geq 1}\frac{b_n}{n^s}\right)
=\sum_{n\geq 1}\left(\sum_{i=1}^{n-1}a_ib_{n-i}\right)\frac{1}{n^s}.
\]
Applied to $\rho^k(s)$ this gives
\[
\rho^k(s)=\sum_{n\geq 2}\left(\sum_{n_1\cdots n_k=n}q^{n_1+\cdots+n_k}\right)
\frac{1}{n^s}.
\]
Hence
\[
\sum_{n\geq 2}\frac{p_n}{n^s}
=(q(q-1)(q+1))^2\sum_{k\geq 1}\left(\frac{q-1}{q}\right)^k\left(\sum_{n\geq 2}
\sum_{n_1\cdots n_k=n}q^{n_1+\cdots+n_k}
\frac{1}{n^s}\right)
\]
\[
=(q(q-1)(q+1))^2\sum_{k\geq 1}\left(\frac{q-1}{q}\right)^k\rho^k(s),
\]
\[
p(s)=\sum_{n\geq 1}\frac{p_n}{n^s}
=p_1+\sum_{n\geq 2}\frac{p_n}{n^s}
\]
\[
=(q(q-1)(q+1))^2\left(-\frac{1}{q+1}+1+\sum_{k\geq 1}\left(\frac{q-1}{q}\right)^k\rho^k(s)\right)
\]
\[
=(q(q-1)(q+1))^2\left(\sum_{k\geq 0}\left(\frac{q-1}{q}\right)^k\rho^k(s)-\frac{1}{q+1}\right)
\]
\[
=(q(q-1)(q+1))^2\left(\frac{1}{1-\frac{q-1}{q}\rho(s)}-\frac{1}{q+1}\right).
\]
\end{proof}

As a consequence of the above theorem, we also give a formula for
the number of coordinates with degree $n$ in ${\mathbb F}_q[x,y]$
and the corresponding Dirichlet series generating function.

\begin{theorem}\label{number of coordinates}
{\rm (i)} The number $l_n$ of coordinates  with degree $n$ in
${\mathbb F}_q[x,y]$, is given by the formulas
\[
l_1=(q-1)q(q+1),
\]
\[
l_n=\frac{p_n}{(q-1)q(q+1)}=q(q-1)(q+1)\sum
\left(\frac{q-1}{q}\right)^kq^{n_1+\cdots+n_k}, \quad n>1,
\]
where the summation takes on all ordered factorizations $n=n_1\cdots
n_k$ of $n$, with $n_1,\ldots,n_k>1$.

{\rm (ii)} The Dirichlet series generating function $l(s)$ of the
sequence $l_n$, $n=1,2,\ldots$, is
\[
l(s)=\sum_{n\geq 1}\frac{l_n}{n^s}
=q(q-1)(q+1)\left(\frac{1}{1-\frac{q-1}{q}\rho(s)}-\frac{1}{q+1}\right),
\]
where
\[
\rho(s)=\sum_{n\geq 2}\frac{q^n}{n^s}.
\]
\end{theorem}

\begin{proof}
{\rm (i)}  According to the well-known theorem of Jung-van der Kulk
\cite{J, K},  for a coordinate $f\in K[x,y]$ with $\deg(f)>1$, two
automorphisms $(f,g)$ and $(f,g_1)$ with $\deg(g)<\deg(f)$ and
$\deg(g_1)<\deg(f)$ if and only if $g_1=cg+d$ where $c\in K-\{0\}$,\
$d\in K$ (so $(c,d)$ has $(q-1)q$ choices),  hence for a fixed
coordinate $f\in\mathbb F_q[x,y]$ with $\deg(f)>1$, there are
$(q-1)q$ automorphisms $(f,g)$ with $\deg(g)<\deg(f)$, therefore
$l_n=\frac{z_n}{(q-1)q}$, where $z_n$ is the number of automorphisms
$(f,g)$ with $\deg(f)=n>\deg(g)$.

Now we can determine $z_n$ as follows. Tracing back to the proof of
Theorem \ref{explicit formula and Dirichlet series}, we can see in
the current case $\alpha_1^{\delta}=(y,x+a_1y)^{\delta}$ only has
one choice (i.e. $\delta=0$) instead of $(q+1)$ choices in the
decomposition as we have $\deg(f)>\deg(g)$ now, so
$z_n=\frac{p_n}{q+1}$. Therefore,

$$l_n=\frac{z_n}{(q-1)q}
=\frac{1}{(q-1)q(q+1)}p_n$$

$$=q(q-1)(q+1)\sum
\left(\frac{q-1}{q}\right)^kq^{n_1+\cdots+n_k}.$$

 When $n=1$, for coordinates $ax+by+c\ (a, b, c\in {\mathbb F}_q,\  (a,b)\ne (0,0))$,\ we have $(q-1)q$ choices
 for $(a, b)$,\ $q$ choices for $c$, hence
$$l_1=q(q^2-1)=(q-1)q(q+1).$$

{\rm (ii)} It follows from {\rm (i)}, and  by  similar calculation
in the proof of Theorem \ref{explicit formula and Dirichlet series}
{\rm (ii)}.
\end{proof}

The following theorem gives an estimate for the growth of $p_n$.

\begin{theorem}\label{estimate for number of automorphisms}
For $n\geq 2$ the number of automorphisms of degree $n$ of ${\mathbb F}_q[x,y]$
satisfies the inequalities
\[
(q-1)^3(q+1)^2q^{n+1}\leq p_n\leq (q-1)^3(q+1)^2q^{n+1}+(\log_2n)^{\log_2n}q^{n/2+8}.
\]
For a fixed $q$ and any $\varepsilon>0$,
\[
p_n=(q-1)^3(q+1)^2q^{n+1}+{\mathcal O}(q^{n(1/2+\varepsilon)}).
\]
\end{theorem}

\begin{proof}
By Theorem \ref{explicit formula and Dirichlet series} (i),
for $n\geq 2$
\[
p_n=(q(q-1)(q+1))^2\sum \left(\frac{q-1}{q}\right)^kq^{n_1+\cdots+n_k},
\]
where the summation is on all ordered factorizations $n=n_1\cdots n_k$
of $n$, with $n_1,\ldots,n_k>1$. For $k=1$ we obtain the summand
$(q-1)^3(q+1)^2q^{n+1}$. Hence it is sufficient to show that
for the number $f(n)$ of ordered factorizations $n=n_1\cdots n_k$ with $k\geq 2$
satisfies
\[
f(n)\leq (\log_2n)^{\log_2n}={\mathcal O}(n^{\varepsilon}),
\]
and for $k\geq 2$
\[
(q(q-1)(q+1))^2\left(\frac{q-1}{q}\right)^kq^{n_1+\cdots+n_k}\leq q^{n/2+8}.
\]
If $n=p_1\cdots p_m$ is the factorization of $n$ in primes, then
the number $k$ in the ordered factorizations is bounded by $m$. The
number of factorizations $f(n)$ is bounded by the
number of factorizations $n=n_1\cdots n_m$ in $m$ parts,
allowing $n_i=1$ for some $i$. Hence $f(n)\leq m^m$
(because each $p_j$ may participate as a factor of any $n_i$). Since
\[
2^m\leq p_1\cdots p_m=n,\quad m\leq\log_2n,
\]
we derive the inequality for $f(n)$. Let $\varepsilon>0$.
Writing $n$ in the form $n=2^t$, $t=\log_2n$,
we obtain $t^t=\log_2n^{\log_2n}$. For $a>1$ and for $n$
sufficiently large (and hence $t$ sufficiently large)
\[
\log_af(n)\leq\log_a(t^t)=t\log_at
\]
\[
\leq t^2\leq 2^t\varepsilon\log_aq
=\log_a\left(q^{2^t\varepsilon}\right)=\log_a(q^{n\varepsilon})
\]
which gives the estimate for $f(n)$.
For the second inequality, we have
\[
(q(q-1)(q+1))^2=q^2(q^2-1)^2<q^6,\quad \frac{q-1}{q}<1,
\]
hence we have to show that
\[
n_1+\cdots+n_k\leq \frac{n}{2}+2
\]
for $n=n_1\cdots n_k$, where $n_i\geq 2$ and $k\geq 2$.
We consider the function
\[
u(t_1,\ldots,t_{k-1})=t_1+\cdots+t_{k-1}+\frac{n}{t_1\cdots n_{k-1}},
\]
\[
2\leq t_1\leq \cdots\leq t_{k-1}\leq t_k=\frac{n}{t_1\cdots t_{k-1}}.
\]
Hence
\[
t_1(t_1\cdots t_{k-1})\leq n,
\]
\[
\frac{\partial u}{\partial t_1}=1-\frac{n}{t_1(t_1\cdots t_{k-1})}\leq 0
\]
with $\partial u/\partial t_1=0$ for $t_1=\cdots=t_{k-1}=n/k$ only.
Considered as a function of $t_1$, the function $u(t_1,\ldots,u_{k-1})$
decreases for $t_1\in [2,t_2]$
and has its maximal value for $t_1=2$.
Hence
\[
u(t_1,\ldots,t_{k-1})=t_1+t_2+\cdots+t_k\leq 2+(t_2+\cdots+t_k),\quad
t_2\cdots t_k=\frac{n}{2}.
\]
If $k=2$ we already have
\[
t_1+t_2\leq \frac{n}{2}+2.
\]
For $k\geq 3$ we have $n\geq 8$, $n/4\geq 2$ and by induction
\[
t_2+\cdots+t_k\leq\frac{n}{2\cdot 2}+2=\frac{n}{4}+2,
\]
\[
t_1+(t_2+\cdots+t_k)\leq 2+\left(\frac{n}{4}+2\right)
\leq \frac{n}{4}+\left(\frac{n}{4}+2\right)=\frac{n}{2}+2.
\]
\end{proof}

The main contribution
$(q-1)^3(q+1)^2q^{n+1}$ to $p_n$ comes from
the case when $k=1$.
By Proposition
\ref{canonical form of automorphisms} this means that
such automorphisms are of the form
\[
\varphi=(f,g)=\alpha\circ\beta\circ\lambda,
\]
where $\alpha=\iota=(x,y)$ or $\alpha=(y,x+ay)$, $a\in {\mathbb F}_q$,
$\beta=(x+h(y),y)$,
$h(y)\in y^2{\mathbb F}_q[y]$ is a polynomial of degree $n$ and $\lambda$ is an affine
automorphism, i.e., $\varphi$ has the form from the introduction.

\

\section{Free Nielsen -- Schreier algebras}

Recall that a variety of algebras over a field $K$ is the class of
all (maybe nonassociative) algebras satisfying a given system of
polynomial identities. Examples of varieties are the classes of all
commutative-associative algebras, all associative algebras, all Lie
algebras, all nonassociative algebras, etc. The variety satisfies
the Nielsen -- Schreier property if the subalgebras of its free
algebras are free in the same variety. See, for instance,
\cite{MSY}. If a Nielsen -- Schreier variety is defined by a
homogeneous (with respect to each variable) system of polynomial
identities, then the automorphisms of the finitely generated free
algebras are tame, see Lewin \cite{L}. We have the following
analogue of Theorem \ref{explicit formula and Dirichlet series}:

\begin{theorem}\label{number of automorphisms for free algebras}
Let $F(x,y)$ be the free ${\mathbb F}_q$-algebra with two generators
in a Nielsen -- Schreier variety defined by a homogeneous system of
polynomial identities and let $F(x,y)\not=0$.
Let $c_n$ be the dimension of all homogeneous polynomials
$u(x)$ in one variable of degree $n$ in $F(x,y)$.

{\rm (i)} The number $p_n$ of automorphisms $\varphi=(f,g)$ of degree $n$ of
$F(x,y)$ is given by the formulas
\[
p_1=q^3(q-1)^2(q+1)\quad \text{for unitary algebras},
\]
\[
p_1=q(q-1)^2(q+1)\quad \text{for nonunitary algebras},
\]
\[
p_n=p_1(q+1)\sum q^{k-1}\prod_{i=1}^k\left((q^{c_{n_i}}-1)q^{c_2+\cdots+c_{n_i-1}}\right),
\quad n>1,
\]
where the summation is on all ordered factorizations $n=n_1\cdots n_k$
of $n$, with $n_1,\ldots,n_k>1$.

{\rm (ii)} The Dirichlet series generating function $p(s)$ of the sequence
$p_n$, $n=1,2,\ldots$, is
\[
p(s)=\frac{p_1}{q}\left(\frac{q+1}{1-q\sigma(s)}-1\right),
\]
where
\[
\sigma(s)=\sum_{n\geq 2}(q^{c_n}-1)q^{c_2+\cdots+c_{n-1}}\frac{1}{n^s}.
\]
\end{theorem}

\begin{proof}
One of the important properties of free algebras in
Nielsen -- Schreier varieties defined by homogeneous polynomial
identities is the following. If several homogeneous elements in the free algebra
are algebraically dependent, then one of them is a polynomial of the others.
This fact implies that the automorphisms of finitely generated free algebras are tame.
Applied to the free algebra $F(x,y)$ with two generators, this gives that
$\text{Aut}(F(x,y))=A\ast_CB$, where $A$ is the affine group
if we consider unitary algebras and the general linear group when we
allow nonunitary algebras, $B$ is the group of triangular automorphisms and $C=A\cap B$.
Hence we have an analogue of Proposition \ref{canonical form of automorphisms}.
Counting the elements of degree $n$ in $\text{Aut}(F(x,y))$ we obtain that
$p_1$ is the number $q^3(q-1)^2(q+1)$ of elements of the affine group
if the algebra $F(x,y)$ has 1 and the number $q(q-1)^2(q+1)$ of elements
of the general linear group $GL_2({\mathbb F}_q)$ for
nonunitary algebras. Then the proof follows the steps of the proof of Theorem
\ref{explicit formula and Dirichlet series}, taking into account that
the number of polynomials $h(y)=h_ny^n+\cdots+h_2y^2$ of degree $n$ is
$(q^{c_n}-1)q^{c_{n-1}}\cdots q^{c_2}$.
\end{proof}

The free Lie algebra and the free anti-commutative algebra in two variables have no
elements in one variable of degree $>1$ and all automorphisms are linear.
Hence
\[
p_1=q(q-1)^2(q+1),\quad p_n=0,\quad n>1.
\]
For the free nonassociative algebra the number $c_n$ is equal to the number of
nonassociative and noncommutative monomials in one variable, or to the Catalan number and
\[
c_n=\frac{1}{n}\binom{2n-2}{n-1},\quad n\geq 1.
\]
No explicit expression is known for the number $c_n$
of nonassociative commutative monomials of degree $n$.

\

\section*{Acknowledgements}

The authors would like to thank Arnaud Bodin for bringing their
attention to the problem and for helpful comments and suggestions.

\


\begin{thebibliography}{ABC}

\bibitem [B]{B}
A. Bodin, {\em Private communication}, May 20, 2008.

\bibitem[C]{C}
P.M. Cohn,
{\em Free Rings and Their Relations},
Second edition, London Mathematical Society Monographs, {\bf 19},
Academic Press, Inc., London, 1985.

\bibitem[Cz]{Cz}
A.J. Czerniakiewicz,
{\em Automorphisms of a free associative algebra of rank 2. I, II},
Trans. Amer. Math. Soc. {\bf 160} (1971), 393-401; {\bf 171} (1972), 309-315.

\bibitem[J]{J}
H.W.E. Jung,
{\em \"{U}ber ganze birationale Transformationen der Ebene},
J. Reine und Angew. Math. {\bf 184} (1942), 161-174.

\bibitem[K]{K}
W. van der Kulk,
{\em On polynomial rings in two variables},
Nieuw Archief voor Wiskunde (3) {\bf 1} (1953), 33-41.

\bibitem[L]{L}
J. Lewin,
{\em On Schreier varieties of linear algebras},
Trans. Amer. Math. Soc. {\bf 132} (1968), 553-562.



\bibitem[MKS]{MKS}
W. Magnus, A. Karrass, D. Solitar,
{\em Combinatorial Group Theory.
Presentations of Groups in Terms of Generators and Relations},
Second edition, Dover Publications, Inc., New York, 1976.

\bibitem[MSY]{MSY}
A.A. Mikhalev, V. Shpilrain, J.-T. Yu, {\em Combinatorial Methods.
Free Groups, Polynomials, and Free Algebras}, CMS Books in
Mathematics/Ouvrages de Math\'ematiques de la SMC, {\bf 19},
Springer-Verlag, New York, 2004.

\bibitem[ML]{ML}
L.G. Makar-Limanov,
{\em On automorphisms of
free algebra with two generators} (Russian),
Funk. Analiz i ego Prilozh. {\bf 4} (1970),  No. 3, 107-108.
Translation: Functional Anal. Appl. {\bf 4} (1970), 262-263.

\bibitem[W]{W}
H.S. Wilf,
{\em  Generatingfunctionology},
Second edition, Academic Press, Inc., Boston, MA, 1994.
Third edition, A K Peters, Ltd., Wellesley, MA, 2006.

\bibitem[Wr]{Wr}
D.Wright,
{\em The amalgamated free product structure of
${\rm GL}\sb{2}(k[X\sb{1},\cdots,X\sb{n}])$
and the weak Jacobian theorem for two variables},
J. Pure Appl. Algebra {\bf 12} (1978), No. 3, 235-251.

\end{thebibliography}
\end{document}